\documentclass[12pt]{amsart}
\usepackage{latexsym, amsmath, amscd, amssymb, amsthm}
\usepackage[T2A]{fontenc}
 \usepackage[cp1251]{inputenc}
\usepackage[english]{babel}
 \newtheorem{lm}{Lemma}
  \newtheorem{te}{Theorem}
	  \newtheorem{co*}{Corollary}
\usepackage{color}

\begin{document}
\setcounter{page}{1}
\title[The Poincar\'e series]{The Poincar\'e series for the algebras of joint invariants and covariants of $n$ linear  forms.}
\author{Nadia Ilash}
\address{Department of Programming,  Computer and Telecommunication systems,
                Khmelnytskyi National  University\\
                Khmelnytskyi, Instytutska,11\\
                29016, Ukraine}
\email{ilashnadya@yandex.ua}

\begin{abstract}   Explicit formulas for computation of the Poincar\'e series for the algebras of joint invariants and covariants of $n$ linear  forms are found. 
Also, for these algebras we calculate the degrees  and asymptotic behavious of the degrees.
\end{abstract}
\maketitle
\noindent
{\bf Keywords:}
classical invariant theory;  invariants; Poincar\`e series; combinatorics \\
{\bf 2010 MSC}:{ \it  13N15; 13A50;  05A19;05E40  } \\

\textbf{1.} Let $V_1$ be the complex vector space  of  linear binary forms  endowed with the natural action of the special linear group  $SL_2$.  Consider the corresponding action of the group  $SL_2$ on  the  algebras of polynomial functions  $\mathbb{C}[n V_{1}]$ and $\mathbb{C}[n V_{1} \oplus \mathbb{C}^2 ],$ where $n V_{1}:=\underbrace{V_1 \oplus V_1 \oplus \cdots \oplus V_1}_{\text{$n$  times}}.$
Denote   by ${\mathcal{I}_{n}=\mathbb{C}[n V_1  ]^{\,SL_2}}$ and by ${\mathcal{C}_{n}=\mathbb{C}[n V_1 \oplus \mathbb{C}^2 ]^{\,SL_2}}$  the  correspon\-ding algebras of   invariant polynomial functions.
 In the language  of classical invariant theory  the algebras  $\mathcal{I}_{n}$ and   $\mathcal{C}_{n}$ are called the algebra  of join invariants and the algebra of join covariants for  the $n$  linear binary forms respectively. 
  A  generating  set of  the algebra $\mathcal{I}_{n}$ was conjectured by Nowicki \cite{Now}. It  had been proved later by different authors,  for instance see \cite{DrM}, \cite{B_nC}.
  The algebras  $\mathcal{C}_{n},$ $\mathcal{I}_{n}$ are   affine graded algebras under the usual degree: 
\begin{gather*}
\mathcal{C}_{n}=(\mathcal{C}_{n})_{0}+(\mathcal{C}_{n})_{1}+\cdots+(\mathcal{C}_{n})_{j}+ \cdots,
\mathcal{I}_{n}=(\mathcal{I}_{n})_{0}+(\mathcal{I}_{n})_{1}+\cdots+(\mathcal{I}_{n})_{j}+ \cdots,
\end{gather*}
where  each of subspaces   $(\mathcal{C}_{n})_{j}$ and $(\mathcal{I}_{n})_{j}$   is   finite-dimensional. The formal power series 
$$
\mathcal{P}(\mathcal{C}_{n},z)=\sum_{j=0}^{\infty }\dim(\mathcal{C}_{n})_{j}\, z^j,
\mathcal{P}(\mathcal{I}_{n},z)=\sum_{j=0}^{\infty }\dim(\mathcal{I}_{n})_{j}\, z^j,
$$
are called the  Poincar\'e series of the algebras  $\mathcal{C}_{n}$
and  $\mathcal{I}_{n}$.
In the paper \cite{B_arh2}  the following expressions for the Poincar\'e series of those algebras was derived: 

	$$\mathcal{P}(\mathcal{I}_{n},z)=\sum_{k=1}^n {\frac{(-1)^{n-k}(n)_{n-k}}{(k-1)!(n-k)!}\frac{d^{k-1}}{dz^{k-1}}\left(\left(\frac{z}{1-z^2}\right)^{2n-k-1}\right)},$$
	$$\mathcal{P}(\mathcal{C}_{n},z)=\sum_{k=1}^n {\frac{(-1)^{n-k}(n)_{n-k}}{(k-1)!(n-k)!}\frac{d^{k-1}}{dz^{k-1}}\left(\frac{(1+z) z^{2n-k-1}}{(1-z^2)^{2n-k}}\right)},$$ 
	where $(n)_m:=n(n+1)\cdots(n+m-1), (n)_0:=1$ denotes the shifted factorial.
	
In the present paper 
 those formulas 
 are reduced to  the following forms:
$$\mathcal{P}(\mathcal{I}_{n},z)=\frac{N_{n-2}(z^2)}{(1-z^2)^{2n-3}} 
 \text{   and  }
\mathcal{P}(\mathcal{C}_{n},z)=\dfrac{W_{n-1}(z^2)+n z N_{{n-1}}(z^2)} {(1-z^2)^{2n-1}},
$$
where 
$$
N_n(z)=\sum_{k=1}^n \frac{1}{k} { n-1 \choose k-1}{ n \choose k-1} z^{k-1} \text{ and } W_n(z)=\sum_{k=0}^{n} {{{n} \choose {k}}^2 z^{k}},
$$
denotes  the \textit{Narayana polynomials}  and 
 the \textit{Narayana polynomials of type B} respectively. 

Also, the degrees  of algebras $\mathcal{I}_{n},$ $\mathcal{C}_{n}$  and asymptotic behaviors of the degrees are calculated using the explicit expressions for the Poincar\'e series.



\textbf{2.} Let us prove several auxiliary combinatorial identities.

\begin{lm}
Let 
$m, k, s $ be non-negative integers. The generalized Le Jen Shoo identity   holds:
$$\sum_{i=0}^{min \{k,m\}} {m \choose i}{m+2s \choose i+s} {k-i+2m+2s \choose 2m+2s}={m+k+s \choose m+s} {m+k+2s \choose m+s}.$$
\end{lm}

\begin{proof}
Taking into account  $$ \displaystyle {{m} \choose i}=0, \text{ for } i>m , \text{ and } \displaystyle {k-i+2m+2s \choose k-i} =0,  \text{ for } i>k,$$ 
we  have 
\begin{gather*} 
\sum_{i=0}^{\infty} {m \choose i}{m+2s \choose i+s} {k-i+2m+2s \choose 2m+2s}=\\
= \sum_{i=0}^k {m \choose i}{m+2s \choose i+s} {k-i+2m+2s \choose 2m+2s} +\sum_{i=k+1}^{\infty} {m \choose i}{m+2s \choose i+s} {k-i+2m+2s \choose 2m+2s} =\\
=\sum_{i=0}^k {m \choose i}{m+2s \choose i+s} {k-i+2m+2s \choose 2m+2s} + \sum_{i=k+1}^{\infty} {m \choose i}{m+2s \choose i+s} \cdot0=\\
= \sum_{i=0}^k {m \choose i}{m+2s \choose i+s} {k-i+2m+2s \choose 2m+2s}=\\
\sum_{i=0}^m {m \choose i}{m+2s \choose i+s} {k-i+2m+2s \choose 2m+2s} +\sum_{i=m+1}^{\infty} {m \choose i}{m+2s \choose i+s} {k-i+2m+2s \choose 2m+2s}=\\
=\sum_{i=0}^m {m \choose i}{m+2s \choose i+s} {k-i+2m+2s \choose 2m+2s} +\sum_{i=m+1}^{\infty} 0 \cdots {m+2s \choose i+s} {k-i+2m+2s \choose 2m+2s}=\\
=\sum_{i=0}^m {m \choose i}{m+2s \choose i+s} {k-i+2m+2s \choose 2m+2s}=\sum_{i=0}^{min \{k,m\}} {m \choose i}{m+2s \choose i+s} {k-i+2m+2s \choose 2m+2s}.
\end{gather*}
Now  the statement  follows immediately from following identity, see  \cite{Sz}: 
 $${a+c+d+e \choose a+c} {b+c+d+e \choose c+e}=\sum_{i} {a+d \choose i+d}{b+c \choose i+c} {a+b+c+d+e-i \choose a+b+c+d},$$
if we set  $a=m, b=m+s,c=s, d=0$ and $e=k. $
\end{proof}

\begin{lm}
Let $k, n>1$  be non-negative integers; then
$$\sum_{i=0}^{\min\{k, n-1\}}(-1)^i { n{+}i{-}1 \choose i}{{n{+}k{-}2} \choose k{-}i} {n{+}2k{-}i{-}1 \choose 2k}= \frac {\displaystyle{n{+}k{-}1\choose k}{n{-}2{+}k\choose k}}{k+1}.$$
\end{lm}
\begin{proof}
We  have 
\begin{align*} 
\sum_{i=0}^{\min\{k, n-1\}}(-1)^i { n+i-1 \choose i}{{n+k-2} \choose k-i} {n+2k-i-1 \choose 2k}=\\=\sum_{i=0}^{n-1}(-1)^i { n+i-1 \choose i}{{n+k-2} \choose k-i} {n+2k-i-1 \choose 2k}=\\=\sum_{i=0}^{k}(-1)^i { n+i-1 \choose i}{{n+k-2} \choose k-i} {n+2k-i-1 \choose 2k}.
\end{align*}
Note that   $${ {n{+}i{-}1} \choose i}{{n{+}k{-}2} \choose k{-}i}=\frac{n{+}i{-}1}{n{-}1} {{n{+}k{-}2} \choose n{+}i-2}{ n{+}i{-}2 \choose n{-}2}= \frac{n{+}i{-}1}{n{-}1} {n{+}k{-}2 \choose n{-}2}{k \choose i},$$ and $$ {\displaystyle \dfrac{n-1}{k+1} {n+k-1\choose k}={n+k-1\choose k+1}}.$$
So we   prove that
$$\sum_{i=0}^{k} (-1)^i (n-1+i) {k \choose i}{n+2k-i-1 \choose 2k} ={n+k-1\choose k+1}.$$
Let us put $S_1=\sum_{i=0}^{k} (-1)^i (n-1+i) {k \choose i}{n+2k-i-1 \choose 2k}.$ We have:
$$S_1=(n-1)\sum_{i=0}^{k} (-1)^i  {k \choose i}{n+2k-i-1 \choose 2k}+ \sum_{i=0}^{k} (-1)^i i {k \choose i}{n+2k-i-1 \choose 2k}.$$
Using the following  identity, see \cite{R}, p.8 
\begin{gather*}
\sum_i (-1)^i {n-i \choose m-i} {p \choose i}= {n-p \choose  m},
\end {gather*}
we get: 
\begin{gather*} 
S_1=(n-1){n+k-1 \choose k}- k\sum_{i=1}^{k} (-1)^{i-1}  {k-1 \choose i-1}{n+2k-(i-1)-2 \choose 2k}=\\=
(n-1){n+k-1 \choose k}- k\sum_{i=0}^{k-1} (-1)^{i}  {k-1 \choose i}{n+2k-2-i \choose n-2-i} =\\= (n-1){n+k-1 \choose k}- k{n+k-1 \choose n-2}=\\
=(n-1)\left(\frac{(n+k-1)!}{k!(n-1)!}-\frac{(n+k-1)!k}{(k+1)!(n-1)!}\right)=\frac{n-1}{k+1}\frac{(n+k-1)!}{(k+1)!(n-2)!}={n+k-1\choose k+1}.
\end{gather*} 
This concludes the proof.
\end {proof}
Substituting $m=n-3$ and $s=1$ into Lemma 1, we obtain:
$$ 
\sum_{i}{{n-3} \choose i}{{n-1} \choose i+1} {{2n+k-i-4} \choose {k-i}} ={n+k-1\choose n-2}{n-2+k\choose n-2}.
$$
Multiplying both sides by  $\dfrac{1}{n-1}, (n>2)$ and using Lemma 2, we get:  
\begin{gather} 
\sum_{i=0}^{\min\{k, n-1\}}(-1)^i { n+i-1 \choose i}{{n+k-2} \choose k-i} {n+2k-i-1 \choose 2k}=
\end{gather}
\begin{gather*}
 =\sum_{i=0}^{\min\{k, n-3\}}{{n-3} \choose i}{{n-2} \choose i} {{2n+k-i-4} \choose {k-i}}\frac{1}{i+1}. 
\end{gather*}



\textbf{3.} We use the  derived above  combinatorial identities to simplify expressions for the Poincar\'e series 
	$\mathcal{P}(\mathcal{I}_{n},z)$ and $\mathcal{P}(\mathcal{C}_{n},z)$ from \cite{B_arh2} .

	\begin{te}
The following formulas hold:
$$
\begin{array}{ll}
(i) &	\mathcal{P}(\mathcal{I}_{n},z)= \frac{\displaystyle \sum_{k=1}^{n-2} {\frac1k {{n-3} \choose {k-1}}{{n-2} \choose {k-1}}z^{2k-2}}} {\displaystyle (1-z^2)^{2n-3}},\\

(ii) & \mathcal{P}(\mathcal{C}_{n},z)=\dfrac{\displaystyle \sum_{k=0}^{n-1} {{{n-1} \choose {k}}^2 z^{2k}}+\sum_{k=0}^{n-2} {{n-2} \choose {k}} {{n} \choose {k+1}} z^{2k+1}} {(1-z^2)^{2n-1}}.
	\end{array}
	$$
\end{te}
\begin {proof}
$(i)$   Let us expand function  $$\sum_{k=1}^n {\frac{(-1)^{n-k}(n)_{n-k}}{(k-1)!(n-k)!}\frac{d^{k-1}}{dz^{k-1}}\left(\left(\frac{z}{1-z^2}\right)^{2n-k-1}\right)},$$ into the  Taylor series about $z$. We have 
\begin{gather*}
\mathcal{P}(\mathcal{I}_{n},z)=\sum_{k=1}^n {\frac{(-1)^{n-k}(n)_{n-k}}{(k-1)!(n-k)!}\frac{d^{k-1}}{dz^{k-1}}\left(z^{2n-k-1}\sum_{i=0}^\infty {{{2n-k+i-2} \choose i} z^{2i}}\right)}=\\=\sum_{k=1}^n {\frac{(-1)^{n-k}(n)_{n-k}}{(k-1)!(n-k)!}\frac{d^{k-1}}{dz^{k-1}}\left(\sum_{i=0}^\infty {{{2n-k+i-2} \choose i} z^{2i+2n-k-1}}\right)}=\\=\sum_{k=1}^n {\frac{(-1)^{n-k}(n)_{n-k}}{(k-1)!(n-k)!}\sum_{i=0}^\infty {{{2n-k+i-2} \choose i} \frac{(2i+2n-k-1)!}{(2i+2n-2k)!}z^{2i+2n-2k}}}.
\end{gather*}
Substituting $j=n-k,$ we have:
 
\begin{gather*}
\mathcal{P}(\mathcal{I}_{n},z)=\sum_{j=0}^{n-1} {\frac{(-1)^{j}(n)_{j}}{(n-j-1)!j!}\sum_{i=0}^\infty {{{n+j+i-2} \choose i} \frac{(2i+n+j-1)!}{(2i+2j)!}z^{2i+2j}}}=\\=\sum_{j=0}^{n-1} {\frac{(-1)^{j}(n+j-1)!}{(n-j-1)!j!(n-1)!}\sum_{i=0}^\infty {{{n+i+j-2} \choose i} \frac{(n+2i+2j-j-1)!}{(2i+2j)!}z^{2i+2j}}}=\\=
\sum_{j=0}^{n-1} {(-1)^{j}{ n+j-1 \choose j} \sum_{i=0}^\infty {{{n+i+j-2} \choose i} {2i+n+j-1 \choose 2i+2j}z^{2i+2j}}}=\\=
\sum_{k=0}^{\infty} \sum_{i=0}^{\min\{k, n-1\}}(-1)^i { n+i-1 \choose i}{{n+k-2} \choose k-i} {n+2k-i-1 \choose 2k}z^{2k}.
\end{gather*}
Using $(1),$ we get:
 
\begin{align*}
  \mathcal{P}(\mathcal{I}_{n},z)=\sum_{k=0}^{\infty} \sum_{i=0}^{\min\{k, n-3\}}{{n-3} \choose i}{{n-2} \choose i} {{2n+k-i-4} \choose {k-i}}\frac{1}{i+1} z^{2k}=\\
=\sum_{k=0}^{n-3}{{n-3} \choose k}{{n-2} \choose k}\frac{z^{2k}}{k+1}\sum_{i=0}^{\infty} {{(2n-3)+i-1} \choose {i}}{z^{2i}}.
\end{align*}
Note that 
  $$\frac{1}{(1-z^2)^{2n-3}}=\sum_{i=0}^\infty {2n-4 +i \choose i} z^{2i}.$$ This completes the proof.

$(ii)$  Denote by  $$A_n (z)=\sum_{k=1}^n {\frac{(-1)^{n-k}(n)_{n-k}}{(k-1)!(n-k)!}\frac{d^{k-1}}{dz^{k-1}}\left(\frac{z^{2n-k-1}}{(1-z^2)^{2n-k}}\right)},$$    and let $
B_n (z)=\mathcal{P}(\mathcal{C}_{n},z)-A_n (z).$
Reasoning as in the proof of $(i),$ we have 
\begin{align*}
A_n (z)=\sum_{k=0}^{\infty} \sum_{i=0}^{\min\{k, n-1\}}(-1)^i { n+i-1 \choose i}{{n+k-1} \choose k-i} {n+2k-i-1 \choose 2k}z^{2k}=\\=\sum_{k=0}^{\infty} { n+k-1 \choose k} z^{2k} \sum_{i=0}^{\min\{k, n-1\}} (-1)^i { k \choose i } { n+2k-i-1 \choose 2k}=\sum_{k=0}^{\infty} { n+k-1 \choose k}^2 z^{2k} .
\end{align*}
   By using  the Le Jen Shoo's identity, we get:
\begin{align*}
A_n (z)=\sum_{k=0}^{\infty} { n-1 \choose i}^2 \sum_{i=0}^{\min\{k, n-1\}}  {{2n+k-i-2} \choose k-i} z^{2k}=\\= \sum_{k=0}^{ n-1}  { n-1 \choose k}^2 z^{2k} \sum_{k=0}^{\infty} { (n-1)+i-1 \choose i} z^{2i}= \frac{\sum_{k=0}^{ n-1}  { n-1 \choose k}^2 z^{2k}}{(1-z^2)^{2n-1} }.
\end{align*}
We see that 
\begin{gather*}
B_n (z)=\sum_{k=1}^n {\frac{(-1)^{n-k}(n)_{n-k}}{(k-1)!(n-k)!}\frac{d^{k-1}}{dz^{k-1}}\left(\frac{z^{2n-k}}{(1-z^2)^{2n-k}}\right)}=\\=\sum_{k=0}^{\infty} \sum_{i=0}^{\min\{k, n-1\}}(-1)^i { n+i-1 \choose i}{{n+k-1} \choose k-i} {n+2k-i \choose 2k+1}z^{2k+1}=\\=\sum_{k=0}^{\infty} { n{+}k{-}1 \choose k} z^{2k{+}1} \sum_{i=0}^{\min\{k, n{-}1\}} ({-}1)^i { k \choose i } { n{+}2k{-}i \choose n{-}1{-}i}=\sum_{k=0}^{\infty} { n{+}k{-}1 \choose n{-}1}{ n{+}k \choose n{-}1} z^{2k+1} .
\end{gather*}
Using lema 1 $(m=n-2, s=1),$ we have:
\begin{gather*}
B_n (z)=\sum_{k=0}^{\infty}  \sum_{i=0}^{\min\{k, n-2\}} { n-2 \choose i } {n \choose i+1}{ k-i+2n-2 \choose 2n-2}=\\=\sum_{k=0}^{n-2} { n-2 \choose k}{ n \choose k+1} z^{2k+1} \sum_{k=0}^{\infty}  { (n-1)+i-1 \choose i} z^{2i}=\frac{\displaystyle  \sum_{k=0}^{ n-2}  { n-2 \choose k}{ n \choose k+1} z^{2k+1} }{(1-z^2)^{2n-1} }.
\end{gather*}
Thus
$$
\mathcal{P}(\mathcal{C}_{n},z)=A_n(z)+B_n (z)=\dfrac{\displaystyle \sum_{k=0}^{n-1} {{{n-1} \choose {k}}^2 z^{2k}}+\sum_{k=0}^{n-2} {{n-2} \choose {k}} {{n} \choose {k+1}} z^{2k+1}} {(1-z^2)^{2n-1}}.
$$
\end {proof}
Let us  rewrite the expressions in terms of the Narayana polynomials $N_n(z)$ and the Narayana polynomials of type B $W_n(z)$
where 
$$
N_n(z)=\sum_{k=1}^n \frac{1}{k} { n-1 \choose k-1}{ n \choose k-1} z^{k-1} \text{ and } W_n(z)=\sum_{k=0}^{n} {{{n} \choose {k}}^2 z^{k}}.
$$
We get
$$\mathcal{P}(\mathcal{I}_{n},z)=\frac{N_{n-2}(z^2)}{(1-z^2)^{2n-3}} 
 \text{      and  }
\mathcal{P}(\mathcal{C}_{n},z)=\dfrac{W_{n-1}(z^2)+n z N_{{n-1}}(z^2)} {(1-z^2)^{2n-1}}.
$$


\textbf{4.} The transcendence degrees over $\mathbb{C}$ for the algebras $\mathcal{I}_{n}, \mathcal{C}_{n}$ is equal to order of the pole  for $\mathcal{P}(\mathcal{I}_{n},z),\mathcal{P}(\mathcal{C}_{n},z)$ respectively, see \cite{SPB}. Note that for all $n$  $N_{n}(1) \neq 0$  and $W_{n}(1)\neq 0.$  These arguments proves 

\begin{te} The following formulas hold
$$
\begin{array}{ll}
(i) & {\rm tr \deg}_{\mathbb{C}}\,\mathcal{I}_{n}= 2n-3,\\
(ii) & {\rm tr \deg}_{\mathbb{C}}\,\mathcal{C}_{n}= 2n-1.
\end{array}
$$
\end{te}

Let  $R= R_0\oplus R_1\oplus \cdots $  be a finitely generated graded complex algebra, $R_0=\mathbb{C}.$ Denote  by
$$
\mathcal{P}(R,z)=\sum_{j=0}^\infty  \dim R_j z^j,
$$
its  Poincar\'e series.
Letting $r$ be  the transcendence degree  of the  quotient field of $R$ over $\mathbb{C},$ the number			
$$
\deg(R):=\lim_{z \to 1} (1-z)^r \mathcal{P}(R,z),
$$ 
is called the  \textit{degree of the algebra} $R$.  The first two terms of  the Laurent series expansion of $\mathcal{P}(R,z)$  at  the point $z=1$ have  the following form 
$$
\mathcal{P}(R,z)=\frac{\deg(R)}{(1-z)^r}+\frac{\psi(R)}{(1-z)^{r-1}}+ \cdots
$$ 
The numbers $\deg(R), \psi(R)$ are important characteristics of the   algebra $R.$ For instance, if
 $R$ is an  algebra of invariants of a finite group  $G$ then  $\deg(R)^{-1}$ is order of the group $G$ and  $2 \dfrac{\psi(R)}{\deg(R)}$ is the number of  pseudo-reflections in $G,$  see \cite{Ben}.

We know explicit forms for the  Poincar\'e series for the algebras of joint invariants and covariants of $n$ linear  forms. Thus we can prove the following statement.
\begin{te}
 The degrees of the algebras of joint invariants and  covariants of $n$ linear  forms are equal to 
 $$
 \begin{array}{ll}
(i) &\deg(\mathcal{P}(\mathcal{I}_{n},z))=\dfrac{N_{n-2}(1)}{2^{2n-3}}=\dfrac{{2n-4\choose n-2}}{(n-1)2^{2n-3}},\\ 

(ii) & 
\deg( \mathcal{P}(\mathcal{C}_{n},z))=\dfrac{{2n-2 \choose n-1}}{2^{2n-2}},
\end{array}
$$
\end{te}
\begin{proof}
$(i)$ Using Theorem 1 and Theorem 2, we have:
\begin{gather*}
\deg(\mathcal{I}_{n})=\lim_{z=1} (1-z)^{2n-3}\mathcal{P}(\mathcal{I}_{n},z)=\lim_{z=1} (1-z)^{2n-3}\dfrac{\displaystyle \sum_{k=1}^{n-2} {\frac1k {{n-3} \choose {k-1}}{{n-2} \choose {k-1}}z^{2k-2}}} {(1-z^2)^{2n-3}}= \\=\frac{N_{n-2}(1)}{2^{2n-3}}
\end{gather*}
Note that the number $N_{n-2}(1)$  equal to  the Catalan numbers, see \cite{Mac}. It now follows that
\begin{gather*}
\deg(\mathcal{I}_{n})=\dfrac{{2n-4\choose n-2}}{(n-1)2^{2n-3}}   
\end{gather*}

$(ii)$ We have
\begin{gather*}
\deg( \mathcal{C}_{n})=\lim_{z=1} (1-z)^{2n-1}\mathcal{P}(\mathcal{C}_{n},z)=\\=\lim_{z=1} (1-z)^{2n-1}\frac{\displaystyle \sum_{k=0}^{n-1} {{{n-1} \choose {k}}^2 z^{2k}}+\sum_{k=0}^{n-2} {{n-2} \choose {k}} {{n} \choose {k+1}} z^{2k+1}} {(1-z^2)^{2n-1}}
= \\=\frac{{2n-2 \choose n-1} +nN_{n-1}(1)}{2^{2n-1}}=\frac{{2n-2 \choose n-1}}{2^{2n-2}}
\end{gather*}
\end  {proof}
Note that asymptotically, the Catalan numbers grow as
$$C_n=\frac{1}{n+1}{2n \choose n} \sim \frac{4^n}{n^{3/2} \sqrt{\pi}}.$$ 
 It is easy to calculate asymptotic behaviours of the degrees of the algebras $ \mathcal{I}_{n}$ і $ \mathcal{C}_{n}$:
\begin {co*} Asymptotic behaviours of the degrees of the algebras of joint invariants and  covariants of $n$ linear  forms as $n\rightarrow\infty$ are follows 
$$
\deg(\text{$ \mathcal{I}_{n}$})\sim \frac{1}{2\sqrt{\pi n^3}}  \text{   {\rm and} } 
\deg( \mathcal{C}_{n}) \sim\frac{1}{\sqrt{\pi n}}.$$

\end {co*}


\end{document}